\documentclass[10pt,english]{smfart}

\usepackage[T1]{fontenc}
\usepackage[english,french]{babel}
\usepackage{latexsym,amscd,color}
\usepackage{amsmath,amsfonts,amssymb,mathrsfs}
\usepackage{enumerate,euscript}
\usepackage{amssymb,url,xspace,smfthm}
\input xy
\xyoption{all}

\DeclareMathOperator{\gal}{Gal}

\newcommand{\BibTeX}{{\scshape Bib}\kern-.08em\TeX}

\newcommand{\T}{\S\kern .15em\relax }
\newcommand{\AMS}{$\mathcal{A}$\kern-.1667em\lower.5ex\hbox
        {$\mathcal{M}$}\kern-.125em$\mathcal{S}$}
\newcommand{\resp}{\textit{resp}.\xspace}

\DeclareMathOperator{\proj}{Proj}

\DeclareMathOperator{\spm}{Spm}

\DeclareMathOperator{\rg}{rk}

\DeclareMathOperator{\spec}{Spec}

\renewcommand{\P}{\mathbb{P}}

\newcommand{\C}{\mathbb{C}}
\newcommand{\Q}{\mathbb{Q}}

\newcommand{\adeg}{\widehat{\deg}}

\newcommand{\p}{\mathfrak{p}}

\newcommand{\E}{\overline{E}}

\newcommand{\sE}{\mathcal{E}}

\renewcommand{\O}{\mathcal{O}}
\renewcommand{\L}{\overline{L}}

\newcommand{\f}{\mathbb{F}}
\newcommand{\ndot}{\raisebox{.4ex}{.}}
\newtheorem*{thm}{Theorem}

\title{Control of the non-geometrically integral reductions}
\alttitle{Contr\^ole des r\'eductions non-g\'eom\'etriquement int\`egres}

\date{\today}
\author{Chunhui Liu}
\address{Institute for Advanced Study in Mathematics\\
Harbin Institute of Technology\\
150001 Harbin\\P. R. China}
\email{chunhui.liu@hit.edu.cn}

\begin{document}
\def\smfbyname{}
\begin{abstract}
In this paper, for a geometrically integral projective scheme, we will give an upper bound of the product of the norms of its non-geometrically integral reductions over an arbitrary number field. For this aim, we take the adelic viewpoint on this subject.
\end{abstract}
\begin{altabstract}
Dans cet article, pour un sch\'ema projectif g\'eom\'etriquement int\`egre, on donnera une majoration du produit des norms de ses r\'eductions non-g\'eom\'etriquement int\`egre sur un corps de nombres arbitraire. Pour le but, on prend le point de vue ad\'elique autour de ce sujet.
\end{altabstract}

\maketitle

\tableofcontents
\section{Introduction}

Let $X\hookrightarrow \mathbb P^n_K\rightarrow\spec K$ be a geometrically integral closed sub-scheme over a number field $K$, $\mathscr X\hookrightarrow\mathbb P^n_{\O_K}\rightarrow\spec\O_K$ be its Zariski closure, and $\mathscr X_{\f_\p}=\mathscr X\times_{\spec\O_K}\spec\f_\p\rightarrow\spec\f_\p$ be its fiber at $\p\in\spm\O_K$. By \cite[Th\'eor\`eme 9.7.7]{EGAIV_3}, the set
\begin{equation}\label{non-geometricaly integral reduction}
  \mathcal Q(\mathscr X)=\left\{\p\in\spm\O_K|\;\mathscr X_{\f_\p}\rightarrow\spec\f_\p\hbox{ is not geometrically integral}\right\}
\end{equation}
is finite.

If we replace the geometrically integral property of $X\rightarrow\spec K$ above by the integral property, the set
\[\mathcal Q'(\mathscr X)=\left\{\p\in\spm\O_K|\;\mathscr X_{\f_\p}\rightarrow\spec\f_\p\hbox{ is not integral}\right\}\]
is not finite in general. For example, we consider
\[X=\proj\left(\mathbb Q[T_0,T_1,T_2]/\left(T_0^2+T_1^2\right)\right)\rightarrow\spec\mathbb Q.\]
 Then we have
\[\mathcal Q'(\mathscr X)=\left\{p\hbox{ prime }|\;p\equiv1(\mathrm{mod}\;4)\hbox{ or }p=2\right\}\]
by properties on the quadric residue, which is a infinite set.

By the above reasons, it is reasonable to give a numerical description of the set of non-geometrically integral reductions $\mathcal Q(\mathscr X)$. Actually, it is easy to construct examples such that $\mathcal Q(\mathscr X)=\emptyset$, for example, the scheme $X$ is a hyperplane in $\mathbb P^n_K$. In addition, we can construct examples whose non-geometrically integral reductions are any products of prime ideals. Let
\[X=\proj\left(\mathbb Q[T_0,T_1,T_2]/\left(T_0^2+aT_1T_2\right)\right)\rightarrow\spec\mathbb Q,\]
where $a\in\mathbb Z$. In this case, we have
\[\mathcal Q(\mathscr X)=\left\{p\hbox{ prime }|\;p\mid a\right\}.\]
 Hence, we are interested in the upper bound of $\sum\limits_{\p\in\mathcal Q(\mathscr X)}\log N(\p)$.

\subsection{Brief history}
Traditionally, we only focus on the case of hypersurfaces when $K=\Q$ and $\O_K=\mathbb Z$, and there are fruitful results reported on this subject. By \cite[Exercise 2.4.1]{LiuQing}, we only need to study whether the polynomial defining this hypersurface is absolutely irreducible over the residue field. Usually, lots of former works considered the case of plane curves only.

Up to the author's knowledge, this subject was first considered by A. Ostrowski in \cite{Ostrowski1919} implicitly. In \cite{Schmidt1976_LNM536}, W. M. Schmidt gave an explicit estimate, which is refined by E. Kaltofen in \cite{Kaltofen1985} (see also \cite{Kaltofen1995}).

In \cite{Ruppert1986}, W. M. Ruppert transferred the criterion of absolute irreducibility of polynomials into the existence of certain polynomial solutions to a certain system of partial differential equations, where he considered the de Rham cohomology of some particular complexes. By this result, he gave an upper bound of non-geometrically reductions for the case of arbitrary hypersurfaces, and a sharper upper bound for the case of plane curves. This result improved some previous results, and was generalized by \cite{Shaker2009} and \cite{GaoShuhong2003} to different directions.

In \cite{Zannier1997}, U. Zannier gave an upper bound depending on the multi-degree of a polynomial $f(x,y)$ over $\mathbb Z$. This result is improved by W. M. Ruppert in \cite{Ruppert1999} by refining his method in \cite{Ruppert1986}. In \cite{Ruppert1999AM}, he considered a special kind of plane curves and gave a better upper bound.

In \cite{GaoRodrigues2003}, Shuhong Gao and V. M. Rodrigues applied Newton polytopes to refine the estimate in \cite{Ruppert1999}, where they involved the number of integral points of Newton polytopes into the estimate.
\subsection{Relation to the arithmetic Hodge index theorem}
In \cite[\S 5]{Faltings84}, G. Faltings proved an arithmetic analogue of the Hodge index theorem. Let $\mathscr X\rightarrow\spec\O_K$, be an arithmetic surface, which means it is flat and projective with the relative dimension $1$, whose generic fiber is $X\rightarrow \spec K$. In \cite[Theorem 4 d)]{Faltings84}, the author gave an estimate of
\[\sum_{\p\in \spm\O_K}\left(\left(\hbox{number of components of }\mathscr X_{\f_\p}\right)-1\right),\]
which is related the rank of jacobian of $X$ over $K$. In fact, the above estimate is able to provide an estimate of the number of non-geometrically integral reductions for the case of curves, while the invariant $\sum\limits_{\p\in\mathcal Q(\mathscr X)}\log N(\p)$ can provide the same estimate. But their methods and fundamental ideas are quite different.
\subsection{Adelic viewpoint}
 In this paper, we will give such an upper bound for the case of an arbitrary number field. Let $X\hookrightarrow\mathbb P^n_K$ be a hypersurface, and we consider the Zariski closure $\mathscr X$ of $X$ in $\mathbb P^n_{\O_K}$. In this case, only when $\O_K$ is a principal ideal domain, $\mathscr X\hookrightarrow\mathbb P^n_{\O_K}$ can always be defined by a primitive equation with coefficients in $\O_K$.

 Similar to the method in \cite{Liu-reduced} to study the non-reduced reductions over an arbitrary number field, we introduce the adelic viewpoint to overcome this obstruction. We consider the polynomial with coefficients in $K$ defining $X\hookrightarrow\mathbb P^n_K$ as coefficients in the adelic ring $\mathbb A_K$ with respect to the diagonal embedding, and then we can obtain a primitive $\mathbb A_{\O_K}$-coefficient polynomial by multiplying an element in $\mathbb A_K$ which does not change the height of polynomial in the adelic sense. Then for each $\p\in\spm\O_K$, the $\p$-part of this primitive polynomial of $\mathbb A_{\O_K}$-coefficients is primitive over $\O_{K,\p}$, which defines $\mathscr X_{\p}\hookrightarrow\mathbb P^n_{\O_{K,\p}}$ from $\mathscr X\hookrightarrow\mathbb P^n_{\O_K}$ via the base change $\spec\O_{K,\p}\rightarrow\spec\O_K$. Then we can consider the reduction type of each $\mathscr X_{\p}$ modulo $\p$.

In order to judge whether a projective hypersurface is geometrically integral, we use a numerical criterion of Ruppert \cite[Satz 3, Satz 4]{Ruppert1986}. For the general case, we use the theory of Chow varieties and Cayley varieties to reduce it to the case of hypersurfaces, which is similar to that in \cite[\S6]{Liu-reduced}. In fact, we have the following estimate in Theorem \ref{non-geometricaly integral reduction of general schemes}.
 \begin{thm}
   Let $X$ be a geometrically integral closed sub-scheme of $\mathbb P^n_K$ of pure dimension $d$ and degree $\delta$, and $\mathscr X$ be the Zariski closure of $X$ in $\mathbb P^n_{\O_K}$. Then we have
   \[\frac{1}{[K:\Q]}\sum_{\p\in\mathcal Q(\mathscr X)}\log N(\p)\leqslant(\delta^2-1)h(X)+C(n,d,\delta),\]
   where $h(X)$ is a height of $X$ and $N(\p)=\#(\O_K/\p)$. We will give the above constant $C(n,d,\delta)$ explicitly in Theorem \ref{non-geometricaly integral reduction of general schemes}, and we have $C(n,d,\delta)\ll_n\delta^3$.
 \end{thm}
 If we consider the case of plane curves ($d=1$ and $n=2$) and use the naive height (see Definition \ref{classic height of hypersurface}) in the above theorem, we are able to obtain $C(n,d,\delta)\ll_n\delta^2\log\delta$ in the above estimate, see Proposition \ref{upper bound of non-geometrically integral reductions of plane curves}.
\subsection{Structure of the article}
This paper is organized as follows. In \S 2, we introduce the useful notions on Diophantine geometry and Arakelov geometry. In \S 3, we recall some results of Ruppert on the criterion of the geometrically integral property. In \S 4, we give an upper bound of the non-geometrically reductions for the case of hypersurfaces by the above results of Ruppert under the adelic viewpoint, and such an upper bound for the general case by applying the theory of Chow varieties and Cayley varieties. In Appendix A, we provide the solution to an exercise in \cite{LiuQing}, which is useful in this work.
\subsection*{Acknowledgement}
I would like to thank Prof. Per Salberger for introducing me the master thesis \cite{Thorarinsson2011} of his former student Stef\'an \TH\'orarinsson, which is a good summary for the previous works on this subject. I would like also to thank the anonymous referee for suggestions on the revision of this paper.
\section{Height functions}
The height of arithmetic varieties is an invariant which evaluates the arithmetic complexity of varieties. In order to study it, we introduce some preliminaries of Arakelov geometry and Diophantine geometry.
\subsection{Normed vector bundles}
  Normed vector bundles are one of the main research objects in Arakelov geometry. Let $K$ be a number field and $\O_K$ be its ring of integers. We denote by $M_K$ the set of places of $K$, by $M_{K,f}$ the set of its finite places and by $M_{K,\infty}$ the set of its infinite places. A \textit{normed vector bundle} on $\spec\O_K$ is a pair $\E=\left(E,\left(\|\ndot\|_v\right)_{v\in M_{K,\infty}}\right)$, where:
  \begin{itemize}
    \item $E$ is a projective $\O_K$-module of finite rank;
    \item $\left(\|\ndot\|_v\right)_{v\in M_{K,\infty}}$ is a family of norms, where $\|\ndot\|_v$ is a norm over $E\otimes_{\O_K,v}\C$ which is invariant under the action of $\gal(\C/K_v)$.
  \end{itemize}

If all the norms $\left(\|\ndot\|_v\right)_{v\in M_{K,\infty}}$ are Hermitian, we call $\E$ a \textit{Hermitian vector bundle} on $\spec\O_K$. In particular, if $\rg_{\O_K}(E)=1$, we say that $\E$ is a \textit{Hermitian line bundle} since all Archimedean norms are Hermitian in this case.
\subsection{Height of arithmetic varieties}

In this part, we introduce a kind of height functions of arithmetic varieties defined via the arithmetic intersection theory developed by Gillet and Soul\'e in \cite{Gillet_Soule-IHES90}, which is first introduced by Faltings in \cite[Definition 2.5]{Faltings91}, see also \cite[III.6]{Soule92}.
\begin{defi}[Arakelov height]\label{arakelov height of projective variety}
Let $K$ be a number field, $\O_K$ be its ring of integers, $\overline{\sE}$ be a Hermitian vector bundle of rank $n+1$ on $\spec\O_K$, and $\overline {\mathcal L}$ be a Hermitian line bundle on $\mathbb P(\sE)$. Let $X$ be a pure dimensional closed sub-scheme of $\mathbb P(\sE_K)$ of dimension $d$, and $\mathscr X$ be the Zariski closure of $X$ in $\mathbb P(\sE)$. The \textit{Arakelov height} of $X$ is defined as the arithmetic intersection number
\begin{equation*}
  \frac{1}{[K:\Q]}\adeg\left(\widehat{c}_1(\overline{\mathcal{L}})^{d+1}\cdot[\mathscr X]\right),
\end{equation*}
where $\widehat{c}_1(\overline{\mathcal{L}})$ is the arithmetic first Chern class of $\overline{\mathcal L}$ (see \cite[Chap. III.4, Proposition 1]{Soule92} for its definition), and $\adeg(\ndot)$ is the \textit{Arakelov degree} of arithmetic cycles. This height is denoted by $h_{\overline{\mathcal{L}}}(X)$ or $h_{\overline{\mathcal{L}}}(\mathscr X)$.
\end{defi}
\subsection{Height of hypersurfaces}
Let $X$ be a hypersurface in $\mathbb P^n_K$. By \cite[Proposition 7.6, Chap. I]{GTM52}, $X$ is defined by a homogeneous polynomial. We define a height function of hypersurfaces by considering its defining polynomial.
\begin{defi}[Naive height]\label{classic height of hypersurface}
Let $f(T_0,\ldots,T_n)=\sum\limits_{(i_0,\ldots,i_n)\in\mathbb N^{n+1}}a_{i_0,\ldots,i_n}T_0^{i_0}\cdots T_n^{i_n}$ be a polynomial. We define the \textit{naive height} of $f(T_0,\ldots,T_n)$ as
\[H_K(f)=\prod_{v\in M_K}\max\limits_{(i_0,\ldots,i_n)\in\mathbb N^{n+1}}\left\{|a_{i_0,\ldots,i_n}|_v\right\}^{[K_v:\Q_v]},\]
and $h(f)=\frac{1}{[K:\Q]}\log H_K(f)$. In addition, if $f(T_0,\ldots,T_n)$ is the defining polynomial of the hypersurface $X\hookrightarrow\mathbb P^n_K$,  we define the \textit{naive height} of $X$ as
\[h(X)=h(f).\]
\end{defi}

\subsection{Adelic height}
In order to consider the reductions over an arbitrary number field, we will introduce the so-called \textit{adelic height} of a polynomial, which has been applied to study the non-reduced reductions in \cite{Liu-reduced}.

Let $K$ be a number field, $\O_K$ be its ring of integers. In addition, we denote by
\[\mathbb A_K=\left\{(a_v)_v\in\prod_{v\in M_K}K_v\mid\;a_v\in\O_{K,v}\hbox{ except a finite number of }v\in M_{K,f}\right\}\]
the adelic ring of $K$, by
 \[\mathbb{A}_{\O_K}=\left\{(a_v)_v\in \mathbb A_K|\;a_v\in\O_{K,v}\hbox{ for all }v\in M_{K,f}\right\}\]
the integral adelic ring of $K$, and by $\Delta:K\hookrightarrow\mathbb A_K$ the diagonal embedding. Let $a=(a_v)_{v\in M_K}\in\mathbb A_K$, we define
\[|a|_{\mathbb A_K}=\prod_{v\in M_K}|a_v|_v^{[K_v:\Q_v]}.\]
\begin{defi}[Local part]\label{p-part}
Let $\{a_{i_0,\ldots,i_n}\}=\{(a^v_{i_0,\ldots,i_n})_{v\in M_K}\}$ be a finite family of elements in $\mathbb A_K$ with the indices $(i_0,\ldots,i_n)\in\mathbb N^{n+1}$, and
\[f(T_0,\ldots,T_n)=\sum_{\begin{subarray}{x}(i_0,\ldots,i_n)\in\mathbb N^{n+1}\end{subarray}}a_{i_0,\ldots,i_n}T_0^{i_0}\cdots T_n^{i_n}\]
be a non-zero polynomial in $\mathbb A_K[T_0,\ldots,T_n]$. For each $v\in M_K$, we denote by
\[f^{(v)}(T_0,\ldots,T_n)=\sum_{\begin{subarray}{x}(i_0,\ldots,i_n)\in\mathbb N^{n+1}\end{subarray}}a^v_{i_0,\ldots,i_n}T_0^{i_0}\cdots T_n^{i_n}\]
the \textit{$v$-part} of $f(T_0,\ldots,T_n)$, or by $f^{(\p)}(T_0,\ldots,T_n)$ for $\p\in\spm\O_K$ corresponding to the place $v\in M_{K,f}$, which is called the \textit{$\p$-part} of $f(T_0,\ldots,T_n)$.
\end{defi}

\begin{defi}[Adelic height]
  Let $f(T_0,\ldots,T_n)=\sum\limits_{\begin{subarray}{c} (i_0,\ldots,i_n)\in\mathbb{N}^{n+1}\end{subarray}}a_{i_0,\ldots,i_n}T_0^{i_0}\cdots T_n^{i_n}$ be a polynomial with coefficients in $\mathbb A_K$, where we denote $a_{i_0,\ldots,i_n}=(a^v_{i_0,\ldots,i_n})_{v\in M_K}\in\mathbb A_K$ for every index $(i_0,\ldots,i_n)$ in the above polynomial. We define
  \[H_{\mathbb A_K}(f)=\prod_{v\in M_{K}}\max_{\begin{subarray}{x}(i_0,\ldots,i_n)\in\mathbb N^{n+1}\end{subarray}}\{|a^v_{i_0,\ldots,i_n}|_v\}^{[K_v:\Q_v]}\]
as the \textit{adelic height} of $f$. In addition, we denote $h(f)=\frac{1}{[K:\Q]}\log H_{\mathbb A_K}(f)$.
\end{defi}
Let $f(T_0,\ldots,T_n)=\sum\limits_{\begin{subarray}{c} (i_0,\ldots,i_n)\in\mathbb{N}^{n+1}\end{subarray}}a_{i_0,\ldots,i_n}T_0^{i_0}\cdots T_n^{i_n}$ be a polynomial with coefficients in $K$, and $c\in\mathbb A_K$ with $|c|_{\mathbb A_K}=1$. Let
 \[g(T_0,\ldots,T_n)=\sum_{\begin{subarray}{x}(i_0,\ldots,i_n)\in\mathbb N^{n+1}\end{subarray}}c\Delta(a_{i_0,\ldots,i_n})T_0^{i_0}\cdots T_n^{i_n}\]
be the polynomial with coefficients in $\mathbb A_K$. Then by definition, we have
\begin{equation}\label{compare adelic height and classic height}
  H_{\mathbb A_K}(g)=H_K(f),
\end{equation}
where $H_K(f)$ is defined in Definition \ref{classic height of hypersurface}.

Let
\[f(T_0,\ldots,T_n)=\sum_{\begin{subarray}{c} (i_0,\ldots,i_n)\in\mathbb{N}^{n+1}\end{subarray}}a_{i_0,\ldots,i_n}T_0^{i_0}\cdots T_n^{i_n}\]
be a polynomial with coefficients in $K$. By \cite[Lemme 3.8]{Liu-reduced}, there exists an element $c\in\mathbb A_K$ with $|c|_{\mathbb A_K}=1$, such that for each $v\in M_{K,f}$, we have
\[\max_{\begin{subarray}{c} (i_0,\ldots,i_n)\in\mathbb{N}^{n+1}\end{subarray}}\{|c\Delta(a_{i_0,\ldots,i_n})|_v\}=1.\]
Let $b_{i_0,\ldots,i_n}=c\Delta(a_{i_0,\ldots,i_n})$, then
\begin{equation}\label{adelicly primitive polynomial}
F(T_0,\ldots,T_n)=\sum_{\begin{subarray}{c} (i_0,\ldots,i_n)\in\mathbb{N}^{n+1}\end{subarray}}b_{i_0,\ldots,i_n}T_0^{i_0}\cdots T_n^{i_n}\in\mathbb A_{\O_K}[T_0,\ldots,T_n],
\end{equation}
which is called an \textit{adelicly primitive polynomial} of $f$.

\section{A criterion of non-geometrically property}
Let $X$ be a geometrically integral hypersurface of $\mathbb P^n_K$ defined by the homogeneous polynomial $f(T_0,\ldots,T_n)$, and $\mathscr X$ be the Zariski closure of $X$ in $\mathbb P^n_{\O_K}$. For all $\p\in\spm\O_K$, in order to study the reduction of $\mathscr X \hookrightarrow\mathbb P^n_{\O_K}\rightarrow\spec\O_K$ at $\p$, we factor the reduction through the localization at $\p$. More precisely, we consider the Cartesian diagram
   \[\xymatrix{\mathscr X_{\f_\p}\ar[r]\ar@{^{(}->}[d] \ar@{}[rd]|{\Box}& \mathscr X_{\O_{K,\p}} \ar[r]\ar@{^{(}->}[d]\ar@{}[rd]|{\Box}& \mathscr X\ar@{^{(}->}[d]\\\mathbb P^n_{\f_\p}\ar[r]\ar[d] \ar@{}[rd]|{\Box}& \mathbb P^n_{\O_{K,\p}} \ar[r]\ar[d]\ar@{}[rd]|{\Box}& \mathbb P_{\O_K}^n\ar[d]\\ \spec \f_\p\ar[r]&\spec\O_{K,\p} \ar[r]&\spec\O_K.}\]
   By definition, $\mathscr X_{\O_{K,\p}}\hookrightarrow\mathbb P^n_{\O_{K,\p}}$ is defined by the $\p$-part $F^{(\p)}(T_0\,\ldots,T_n)$ of $F(T_0,\ldots,T_n)$ (see Definition \ref{p-part} for its definition, which is primitive over $\O_{K,\p}$ by the construction of $F(T_0,\ldots,T_n)$ in \eqref{adelicly primitive polynomial}.

By \cite[Exercise 2.4.1]{LiuQing} (see \cite[Remarque 5.2]{Liu-reduced} for a projective version), for an arbitrary $\p\in\spm\O_K$, the fact that the polynomial $F^{(\p)}(T_0\,\ldots,T_n)$ modulo $\p[T_0,\ldots,T_n]$ is not absolutely irreducible over $\f_\p$ is verified if and only if $\mathscr X_{\f_\p}$ is not geometrically integral over $\spec\f_\p$. So in order to control the set $\mathcal Q(\mathscr X)$ introduced in \eqref{non-geometricaly integral reduction}, we need to study the absolute irreducibility of $F^{(\p)}(T_0,\ldots,T_n)\mod \p[T_0,\ldots,T_n]$ for all $\p\in\spm\O_K$.

The first result is for the case of plane curves.
\begin{prop}[\cite{Ruppert1986}, Satz 3]\label{criterion of irreducibility of curves}
  Let
\[g(T_0,T_1,T_2)=\sum_{\begin{subarray}{c} (i_0,i_1,i_2)\in\mathbb{N}^{3}\\ i_0+i_1+i_2=\delta\end{subarray}}b_{i_0,i_1,i_2}T_0^{i_0}T_1^{i_1} T_2^{i_2}\]
be a homogeneous polynomial of degree $\delta$ over an algebraically closed field $k$. Then there exists a family of homogeneous polynomial $\{\phi_j\}_{j\in J}\in\mathbb Z[b_{i_0,i_1,i_2}]$ with the index set $J$ and variables $\{b_{i_0,i_1,i_2}|\;(i_0,i_1,i_2)\in\mathbb N^{3},\;i_0+i_1+i_2=\delta\}$, which are of degree $\delta^2-1$ and length smaller than $\delta^{3\delta^2-3}$, such that
\begin{enumerate}
  \item If $F$ is reducible, then $\phi_j(b_{i_0,i_1,i_2})=0$ for every $j\in J$;
  \item If $F$ is irreducible and $k$ is of characteristic $0$, then there exists at least one $j\in J$, such that $\phi_j(b_{i_0,i_1,i_2})\neq0$.
\end{enumerate}
\end{prop}
The second one is for the case of general hypersurfaces.
\begin{prop}[\cite{Ruppert1986}, Satz 4]\label{criterion of irreducibility}
  Let
\[g(T_0,\ldots,T_n)=\sum_{\begin{subarray}{c} (i_0,\ldots,i_n)\in\mathbb{N}^{n+1}\\ i_0+\cdots+i_n=\delta\end{subarray}}b_{i_0,\ldots,i_n}T_0^{i_0}\cdots T_n^{i_n}\]
be a homogeneous polynomial of degree $\delta$ over an algebraically closed field $k$. Then there exists a family of homogeneous polynomial $\{\phi_j\}_{j\in J}\in\mathbb Z[b_{i_0,\ldots,i_n}]$ with the index set $J$ and variables $\{b_{i_0,\ldots,i_n}|\;(i_0,\ldots,i_n)\in\mathbb N^{n+1},\;i_0+\cdots+i_n=\delta\}$, which are of degree $\delta^2-1$ and length smaller than $\delta^{3\delta^2-3}\left[{n+\delta\choose \delta}3^\delta\right]^{\delta^2-1}$, such that
\begin{enumerate}
  \item If $F$ is reducible, then $\phi_j(b_{i_0,\ldots,i_n})=0$ for every $j\in J$;
  \item If $F$ is irreducible and $k$ is of characteristic $0$, then there exists at least one $j\in J$, such that $\phi_j(b_{i_0,\ldots,i_n})\neq0$.
\end{enumerate}
\end{prop}

\section{Control of the non-geometrically integral reductions}
By Proposition \ref{criterion of irreducibility of curves} and \ref{criterion of irreducibility}, Ruppert gives a control of non-geometrically integral reductions of hypersurfaces in $\mathbb P^n_{\mathbb Z}$ in \cite[Korollar 1, Korollar 2]{Ruppert1986}. In this part, we will give such a control over an arbitrary number field $K$ for general projective schemes.
\subsection{Non-geometrically integral reductions of hypersurfaces}
For the case of hypersurfaces, by applying Proposition \ref{criterion of irreducibility of curves} and \ref{criterion of irreducibility} to an adelicly primitive polynomial introduced at \eqref{adelicly primitive polynomial}, we have the following two results. Since their proofs are quite similar, we only provide the detailed proof for the case of general hypersurfaces.
\begin{prop}\label{upper bound of non-geometrically integral reductions of hypersurfaces}
  Let $X$ be a geometrically integral hypersurface in $\mathbb P^n_K$ of degree $\delta$, $\mathscr X$ be its Zariski closure in $\mathbb P^n_{\O_K}$, $\mathscr X_{\f_\p}=\mathscr X\times_{\spec\O_K}\spec\f_\p$, and
\[\mathcal Q(\mathscr X)=\left\{\p\in\spm\O_K|\;\mathscr X_{\f_\p}\rightarrow\spec\f_\p\hbox{ is not geometrically integral}\right\}.\]
 Then we have
  \begin{equation*}
  \frac{1}{[K:\Q]}\sum_{\p\in\mathcal Q(\mathscr X)}\log N(\p)\leqslant(\delta^2-1)h(X)+C(n,\delta),
\end{equation*}
  where $N(\p)=\#(\O_K/\p)$, $h(X)$ is the classic height of $X$ in $\mathbb P^n_K$ defined in Definition \ref{classic height of hypersurface}, and the constant
  \[C(n,\delta)=(\delta^2-1)\left(3\log\delta+\delta\log3+\log{n+\delta\choose\delta}\right).\]
\end{prop}
\begin{proof}
  Suppose $X$ is defined by the homogeneous polynomial
  \[f(T_0,\ldots,T_n)=\sum_{\begin{subarray}{c} (i_0,\ldots,i_n)\in\mathbb{N}^{n+1}\\i_0+\cdots+i_n=\delta\end{subarray}}a_{i_0,\ldots,i_n}T_0^{i_0}\cdots T_n^{i_n}\]
  with coefficients in $K$, and
   \begin{equation*}
F(T_0,\ldots,T_n)=\sum_{\begin{subarray}{c} (i_0,\ldots,i_n)\in\mathbb{N}^{n+1}\\i_0+\cdots+i_n=\delta\end{subarray}}b_{i_0,\ldots,i_n}T_0^{i_0}\cdots T_n^{i_n}
\end{equation*}
 be an adelicly primitive polynomial of $f(T_0,\ldots,T_n)$ constructed in \eqref{adelicly primitive polynomial}.
We use the notations in Proposition \ref{criterion of irreducibility}, and choose an index $j\in J$ of the polynomial $\phi_j(b_{i_0,\ldots,i_n})$ with variables $b_{i_0,\ldots,i_n}$, such that $\phi_j(a_{i_0,\ldots,i_n})\neq0$ for the coefficients of $f(T_0,\ldots,T_n)$.

  For each $\p\in\spm\O_K$, since $b_{i_0,\ldots,i_n}^{(\p)}\in\O_{K,\p}$, we have $\left|\phi_j\left(b_{i_0,\ldots,i_n}^{(\p)}\right)\right|_{\p}\leqslant1$ if $\phi_j\left(b_{i_0,\ldots,i_n}^{(\p)}\right)\neq0$. By definition, if the maximal ideal $\p\in\mathcal Q(\mathscr X)$, we have $\left|\phi_j\left(b_{i_0,\ldots,i_n}^{(\p)}\right)\right|_{\p}<1$. Then we obtain
  \begin{eqnarray*}
    \frac{1}{[K:\Q]}\sum_{\p\in\mathcal Q(\mathscr X)}\log N(\p)&\leqslant&-\sum_{\p\in\mathcal Q(\mathscr X)}\frac{[K_\p:\Q_\p]}{[K:\Q]}\log\left(\left|\phi_j\left(b_{i_0,\ldots,i_n}^{(\p)}\right)\right|_{\p}\right)\\
    &\leqslant&-\sum_{\p\in\spm\O_K}\frac{[K_\p:\Q_\p]}{[K:\Q]}\log\left(\left|\phi_j\left(b_{i_0,\ldots,i_n}^{(\p)}\right)\right|_{\p}\right)\\
    &=&\frac{1}{[K:\Q]}\sum_{v\in M_{K,\infty}}\log\left(\left|\phi_j\left(b_{i_0,\ldots,i_n}^{(v)}\right)\right|_{v}\right).
  \end{eqnarray*}
  In order to estimate $\log\left(\left|\phi_j\left(b_{i_0,\ldots,i_n}^{(v)}\right)\right|_{v}\right)$ for a fixed $v\in M_{K,\infty}$, from the properties of $\phi_j$ given in Proposition \ref{criterion of irreducibility}, we have
    \begin{eqnarray}\label{norm of universal poly}
      \log\left(\left|\phi_j\left(b_{i_0,\ldots,i_n}^{(v)}\right)\right|_{v}\right)&\leqslant&(\delta^2-1)\log\left(\max_{\begin{subarray}{c} (i_0,\ldots,i_n)\in\mathbb{N}^{n+1}\\i_0+\cdots+i_n=\delta\end{subarray}}\left\{|b^{(v)}_{i_0,\ldots,i_n}|_v\right\}\right)\\
      & &\;+(\delta^2-1)\left(3\log\delta+\delta\log3+\log{n+\delta\choose\delta}\right).\nonumber
    \end{eqnarray}
Then from \eqref{norm of universal poly}, we obtain
\begin{eqnarray*}
  & &\frac{1}{[K:\Q]}\sum_{v\in M_{K,\infty}}\log\left(\left|\phi_j\left(b_{i_0,\ldots,i_n}^{(v)}\right)\right|_{v}\right)\\
  &\leqslant&\frac{\delta^2-1}{[K:\Q]}\sum_{v\in M_{K,\infty}}\log\left(\max_{\begin{subarray}{c} (i_0,\ldots,i_n)\in\mathbb{N}^{n+1}\\i_0+\cdots+i_n=\delta\end{subarray}}\left\{|b^{(v)}_{i_0,\ldots,i_n}|_v\right\}\right)\\
      & &\;+(\delta^2-1)\left(3\log\delta+\delta\log3+\log{n+\delta\choose\delta}\right)\\
      &=&(\delta^2-1)h(X)+(\delta^2-1)\left(3\log\delta+\delta\log3+\log{n+\delta\choose\delta}\right),
\end{eqnarray*}
where the last equality is from \eqref{compare adelic height and classic height} and \eqref{adelicly primitive polynomial}. Then we have the assertion.
\end{proof}
\begin{rema}
With all the notations in Proposition \ref{upper bound of non-geometrically integral reductions of hypersurfaces}, we have $C(n,\delta)\ll_{n}\delta^3$.
\end{rema}
By applying Proposition \ref{criterion of irreducibility of curves} to the proof of Proposition \ref{upper bound of non-geometrically integral reductions of hypersurfaces}, we have the following estimate for plane curves, where we will only point out the key difference in the proofs.
\begin{prop}\label{upper bound of non-geometrically integral reductions of plane curves}
  Let $X$ be a geometrically integral plane curve in $\mathbb P^2_K$ of degree $\delta$, $\mathscr X$ be its Zariski closure in $\mathbb P^2_{\O_K}$, $\mathscr X_{\f_\p}=\mathscr X\times_{\spec\O_K}\spec\f_\p$, and
\[\mathcal Q(\mathscr X)=\left\{\p\in\spm\O_K|\;\mathscr X_{\f_\p}\rightarrow\spec\f_\p\hbox{ is not geometrically integral}\right\}.\]
Then we have
  \begin{equation*}
  \frac{1}{[K:\Q]}\sum_{\p\in\mathcal Q(\mathscr X)}\log N(\p)\leqslant(\delta^2-1)h(X)+C(\delta),
\end{equation*}
  where $N(\p)=\#(\O_K/\p)$, $h(X)$ is the classic height of $X$ in $\mathbb P^n_K$ defined in Definition \ref{classic height of hypersurface}, and the constant $C(\delta)=(3\delta^2-3)\log\delta$.
\end{prop}
\begin{proof}[Sketch of the proof]
 We replace \eqref{norm of universal poly} by the upper bound of the length in Proposition \ref{criterion of irreducibility of curves} in the proof of Proposition \ref{upper bound of non-geometrically integral reductions of hypersurfaces}, then we prove the assertion.
\end{proof}
\begin{rema}
With all the notations in Proposition \ref{upper bound of non-geometrically integral reductions of plane curves}, we have $C(\delta)\ll\delta^2\log\delta$, which has a better dependence on the degree than the case of general hypersurfaces provided in Proposition \ref{upper bound of non-geometrically integral reductions of hypersurfaces}. If we only consider the dependence on the degree of plane curves, this estimate has the same as the later improvements.
\end{rema}
\subsection{Non-geometrically reductions of general projective schemes}
In order to study the non-geometrically reductions of general schemes, it is significant to understand the reductions over their Chow varieties or Cayley varieties. Then we will reduce the general case to that of hypersurfaces. In this paper, will only use Cayley varieties, and Chow varieties are only mentioned for a historical reason.
\subsubsection{Cayley variety}
First, we briefly recall the construction of Cayley varieties. For more details applied in the quantitative arithmetics, we refer the readers to \cite[\S3]{Chen1}, see also \cite[\S2]{Liu-reduced} for the application to the study of the non-reduced reductions.

Let $A$ be a Dedekind domain or a field, $\sE$ be a vector bundle of rank $n+1$ over $\spec A$, and $d\in\mathbb N$ satisfying $1\leqslant d\leqslant n-1$. We denote
\[\theta: \sE^\vee\otimes_A\left(\wedge^{d+1}\sE\right)\rightarrow\wedge^d\sE\]
the homomorphism which maps $\xi\otimes(x_0\wedge\cdots\wedge x_n)$ to
\[\sum_{i=0}^d(-1)^i\xi(x_i)x_0\wedge\cdots\wedge x_{i-1}\wedge x_{i+1}\wedge\cdots\wedge x_d.\]
Let $\Gamma$ be the sub-variety of $\mathbb P(\sE)\times_{\spec A}\mathbb P(\wedge^{d+1}\sE^\vee)$ which classifies the all the points $(\xi,\alpha)$ such that $\theta(\xi\otimes\alpha)=0$. Let $p:\mathbb P(\sE)\times_{\spec A}\mathbb P\left(\wedge^{d+1}\sE^\vee\right)\rightarrow\mathbb P(\sE)$ and $q:\mathbb P(\sE)\times_{\spec A}\mathbb P\left(\wedge^{d+1}\sE^\vee\right)\rightarrow\mathbb P\left(\wedge^{d+1}\sE^\vee\right)$ be the two canonical projections.

Next, let $\overline{\sE}$ be a Hermitian vector bundle on $\spec\O_K$, $X$ be a pure dimensional closed sub-scheme of $\mathbb P(\sE_K)$ of dimension $d$ and degree $\delta$, and $\mathscr X$ be the Zariski closure of $X$ in $\mathbb P(\sE)$. By \cite[Proposition 3.4]{Chen1} or \cite[Proposition 2.2, Proposition 2.7]{Liu-reduced}, the scheme $q(\Gamma\cap p^{-1}(X))$ (\resp $q(\Gamma\cap p^{-1}(\mathscr X))$) is a geometrically integral hypersurface in $\mathbb P\left(\wedge^{d+1}\sE_K^\vee\right)$ (\resp $\mathbb P\left(\wedge^{d+1}\sE^\vee\right)$), and $q(\Gamma\cap p^{-1}(X))$ is of degree $\delta$. We call these hypersurfaces the \textit{Cayley varieties} of $X$ and $\mathscr X$. We denote by $\Psi_{X}\hookrightarrow\mathbb P\left(\wedge^{d+1}\sE_K^\vee\right)$ and $\Psi_{\mathscr X}\hookrightarrow\mathbb P\left(\wedge^{d+1}\sE^\vee\right)$ the Cayley varieties of $X$ and $\mathscr X$ respectively.

By \cite[\S4.3.2 (i), (iv)]{BGS94}, the construction of Cayley varieties commutes with the extension from $X\hookrightarrow\mathbb P(\sE_K)$ to $\mathscr X\hookrightarrow\mathbb P(\sE)$, and commutes with the base change from $\O_K$ to its residue field, see \cite[\S4.3.1]{BGS94} or \cite[Proposition 2.7]{Liu-reduced} for more details of the above argument. Then in order to control the non-geometrically integral reductions of $\mathscr X\rightarrow\spec\O_K$, we are able to consider the non-geometrically reductions of its Cayley varieites.
\subsubsection{Control of the non-geometrically integral reductions}\label{general non-geometrically reductions}
With the above constructions, we consider the non-geometrically reductions of general projective schemes below. We pick $\overline{\sE}=\left(\O_K^{\oplus(n+1)},(\|\ndot\|_v)_{v\in M_{K,\infty}}\right)$, where for each $v\in M_{K,\infty}$, the norm $\|\ndot\|_v$ maps $(x_0,\ldots,x_n)$ to $\sqrt{|x_0|^2_v+\cdots+|x_n|^2_v}$. In this case, we denote $\mathbb P(\sE_K)$ and $\mathbb P(\sE)$ by $\mathbb P^n_K$ and $\mathbb P^n_{\O_K}$ respectively for simplicity.
\begin{theo}\label{non-geometricaly integral reduction of general schemes}
  With all the above notations and conditions in \S \ref{general non-geometrically reductions}. Let $X$ be a geometrically integral closed sub-scheme of $\mathbb P^n_K$ of pure dimension $d$ and degree $\delta$, $\mathscr X$ be the Zariski closure of $X$ in $\mathbb P^n_{\O_K}$, $\mathscr X_{\f_\p}=\mathscr X\times_{\spec\O_K}\spec\f_\p$, and
  \[\mathcal Q(\mathscr X)=\left\{\p\in\spm\O_K|\;\mathscr X_{\f_\p}\rightarrow\spec\f_\p\hbox{ is not geometrically integral}\right\}.\]
We denote $N(n,d)={n+1\choose d+1}-1$, $N(\p)=\#(\O_K/\p)$, and $\mathcal H_m=1+\cdots+\frac{1}{m}$. Then we have
\begin{equation*}
  \frac{1}{[K:\Q]}\sum_{\p\in\mathcal Q(\mathscr X)}\log N(\p)\leqslant(\delta^2-1)h_{\overline{\O(1)}}(X)+C'(n,d,\delta),
\end{equation*}
  where $\overline{\O(1)}$ is equipped with the corresponding Fubini-Study metrics for all $v\in M_{K,\infty}$, $h_{\overline{\O(1)}}(X)$ is the Arakelov height of $X$ in $\mathbb P^n_K$ defined in Definition \ref{arakelov height of projective variety}, and the constant
    \begin{eqnarray*}C'(n,d,\delta)&=&(\delta^2-1)\Bigg(3\log\delta+\log{N(n,d)+\delta\choose\delta}+\\
  & &\left((N(n,d)+1)\log 2+4\log(N(n,d)+1)+\log3-\frac{1}{2}\mathcal H_{N(n,d)}\right)\delta\Bigg).
  \end{eqnarray*}
\end{theo}
\begin{proof}
  Let $\Psi_{\mathscr X}$ be the Cayley variety of $\mathscr X$, and
  \begin{equation*}\mathcal Q(\Psi_{\mathscr X})=\left\{\p\in\spm\O_K|\Psi_{\mathscr X}\times_{\spec\O_K}\spec\f_\p\hbox{ is not geometrically integral}\right\}.\end{equation*}
  Then by \cite[Proposition 2.2, 2.7]{Liu-reduced}, the fact $\p\in\mathcal Q(\Psi_{\mathscr X})$ is verified if and only if $\mathscr X\times_{\spec\O_K}\spec\f_\p\rightarrow\spec\f_\p$ is not geometrically integral. So we obtain
  \[\frac{1}{[K:\Q]}\sum_{\p\in\mathcal Q(\mathscr X)}\log N(\p)=\frac{1}{[K:\Q]}\sum_{\p\in\mathcal Q(\Psi_{\mathscr X})}\log N(\p).\]
  By Proposition \ref{upper bound of non-geometrically integral reductions of hypersurfaces}, we have
  \begin{eqnarray*}
& &\frac{1}{[K:\Q]}\sum_{\p\in\mathcal Q(\Psi_{\mathscr X})}\log N(\p)\\
&\leqslant&(\delta^2-1)h(\Psi_{X})+(\delta^2-1)\left(3\log\delta+\delta\log3+\log{N(n,d)+\delta\choose\delta}\right),
  \end{eqnarray*}
where $h(\Psi_{X})$ is defined in Definition \ref{classic height of hypersurface}.

By \cite[Proposition 3.7]{Liu-reduced}, we have
\[h(\Psi_{X})-h_{\overline{\O(1)}}(X)\leqslant(N(n,d)+1)\delta\log 2+4\delta\log(N(n,d)+1)-\frac{1}{2}\delta\mathcal H_{N(n,d)}.\]
So we obtain the assertion by combining the above estimates.
\end{proof}
\begin{rema}
We consider the constant $C'(n,d,\delta)$ in Theorem \ref{non-geometricaly integral reduction of general schemes}. Then we have $C'(n,d,\delta)\ll_n\delta^3$. Due to the comparison of heights, we have the same estimate of this constant to the case of curves and of general dimensions if we choose the Arakelov height.
\end{rema}
\appendix
\section{A criterion of the reduced and irreducible properties}
In this appendix, we will give a criterion of the reduced and irreducible hypersurfaces, which is exactly a solution to the exercise \cite[Exercise 2.4.1]{LiuQing}. This result is useful to judge whether a scheme is geometrically integral.
\begin{prop}
  Let $k$ be a field and $P\in k[T_1,\ldots,T_n]$. Then the scheme $\spec\left(k[T_1,\ldots,T_n]/(P)\right)$ is reduced (\resp irreducible; \resp integral) if and only if $P$ has no square factor (\resp admits only one irreducible factor; \resp is irreducible).
\end{prop}
\begin{proof}
  First, we consider the reduced property. By \cite[Definition 2.4.1, Proposition 2.4.2 (b)]{LiuQing}, $\spec\left(k[T_1,\ldots,T_n]/(P)\right)$ is reduced if and only if the ring $k[T_1,\ldots,T_n]/(P)$ is reduced. In fact, the nilradical of $k[T_1,\ldots,T_n]/(P)$ is not zero if and only if there exists a non-zero element $f\in k[T_1,\ldots,T_n]$, such that $f\not\in (P)$ but $f^m\in (P)$ for some $m\in\mathbb N_+$, which is verified if and only if $(P)$ has at least one square factor.

  Next, we consider the irreducible property. Let $I=(P)$ be an ideal of $k[T_1,\ldots,T_n]$, and then we have $\spec\left(k[T_1,\ldots,T_n]/(P)\right)=V(I)$. By \cite[Proposition 2.4.7 (a)]{LiuQing}, $V(I)$ is irreducible if and only if $\sqrt{I}$ is prime, which is verified if and only if $P$ admits only one square factor.

  By \cite[Definition 2.4.16]{LiuQing}, a scheme is integral if and only it is both reduced and irreducible. Then we obtain the result about the integral property.
\end{proof}
\backmatter

\bibliography{liu}
\bibliographystyle{smfplain}

\end{document}